\documentclass{amsart}

\usepackage{enumerate}
\usepackage{amsmath, amscd}
\usepackage{ifpdf}
\usepackage{amsfonts}
\usepackage[normalem]{ulem}
\usepackage{color}
\usepackage{amssymb}
\usepackage{amsthm}
\usepackage[hyperfootnotes=false]{hyperref}
\usepackage{setspace}
\usepackage{graphicx}

\newtheorem{thm}{Theorem}
\newtheorem{main theorem}[thm]{Main Theorem}

\newtheorem{lemma}[thm]{Lemma}
\newtheorem{prop}[thm]{Proposition}
\newtheorem*{prop*}{Proposition}

\theoremstyle{definition}

\newtheorem{defn}[thm]{Definition}
\newtheorem{remark}[thm]{Remark}

\newtheorem{example}[thm]{Example}

\newcommand{\bea}{\begin{eqnarray*}}
\newcommand{\eea}{\end{eqnarray*}}

\newcommand{\be}{\begin{equation}}
\newcommand{\ee}{\end{equation}}
\newcommand{\C}{\mathbb{C}}

\newcommand{\N}{\mathbb{N}}

\let\phe=\varphi

\ifpdf
\title[Elliptic polynomial skew products]{Fatou components of elliptic polynomial skew products}

\author[Han Peters, Jasmin Raissy]{Han Peters, Jasmin Raissy}

\thanks{The work of  Jasmin Raissy was partially supported by ANR
 project LAMBDA,  ANR-13-BS01-0002, by the FIRB2012 grant ``Differential Geometry and Geometric Function Theory'', RBFR12W1AQ 002 and CNRS. }

\begin{document}
\bibliographystyle{plain}

\begin{abstract}
We investigate the description of Fatou components for polynomial skew-products in two complex variables. The non-existence of wandering domains near a super-attracting invariant fiber was shown in \cite{Lilov}, and the geometrically-attracting case was studied in \cite{PV} and \cite{PS}. In \cite{ABDPR} it was proven that wandering domains can exist near a parabolic invariant fiber. In this paper we study the remaining case, namely the dynamics near an elliptic invariant fiber. We prove that the two-dimensional Fatou components near the elliptic invariant fiber correspond exactly to the Fatou components of the restriction to the fiber, under the assumption that the multiplier at the elliptic invariant fiber satisfies the Brjuno condition and that the restriction polynomial has no critical points on the Julia set. We also show the description does not hold when the Brjuno condition is dropped. Our main tool is the construction of expanding metrics on nearby fibers, and one of the key steps in this construction is given by a local description of the dynamics near a parabolic periodic cycle.
\end{abstract}

\maketitle

\tableofcontents

\section{Introduction}

Recently there have been several results regarding the description of Fatou components for polynomial skew products, i.e. maps of the form
$$
F(z,w) = (f(z), g(z,w)),
$$
with $f$ and $g$ polynomials. In \cite{ABDPR} the case of the dynamics near a parabolic fixed point of $f$ was considered, and it was shown that wandering Fatou components can occur for such maps. These were the first, and currently the only, examples of wandering Fatou components for polynomial maps in several complex variables.

Earlier Lilov, in \cite{Lilov}, considered the dynamics near a super-attracting fixed point of $f$. Lilov showed that in the super-attracting basin of $f$, the only Fatou components of $F$ are the \emph{bulging} Fatou components, the two-dimensional Fatou components corresponding to the one-dimensional Fatou components of the restriction to the invariant super-attracting fiber. The attracting but not super-attracting case was studied in \cite{PV} and \cite{PS}. While the attracting case is in general still open, under additional assumptions the same characterization was obtained: near the invariant Fiber the only Fatou components are the bulging components.

Here we study the remaining case, namely the dynamics in a Siegel disk of $f$. We note that examples of Fatou components in this setting were studied previously in \cite{BFP}, where invariant Fatou components with punctured limit sets were constructed.

Let us be more precise about our setting. We suppose that $f(0) = 0$ and $f^\prime(0) = \lambda$, with $\lambda = e^{2 \pi i\theta}$ and $\theta \in \mathbb R \setminus \mathbb Q$. We assume that $f$ is linearizable near $z = 0$, so that after a local change of coordinates we may assume that $F$ is of the form
$$
F(z,w) = (\lambda \cdot z, g_z(w)),
$$
where $g$ is a polynomial in $w$ with coefficients depending holomorphically on $z$. We assume that the degree of the polynomial $g_z$ is constant near $z = 0$, and at least $2$. Our main result is the following.

\begin{thm}\label{thm:main}
If $\lambda$ is Brjuno and all critical points of the polynomial $g_0$ lie in basins of attracting or parabolic cycles, then all Fatou components of $g_0$ bulge, and there is a neighborhood of the invariant fiber $\{z=0\}$ in which the only Fatou components of $F$ are the bulging Fatou components of $g_0$. In particular there are no wandering Fatou components in this neighborhood.
\end{thm}

Recall that a complex number is called \emph{Brjuno} (see \cite{Brjuno1} for more details) if
\begin{equation}\label{eq:brjuno}
\sum_{k=0}^{+\infty}{\frac{1}{2^k}}\log{\frac{1}{\omega(2^{k+1})}}<+\infty\;,
\end{equation}
where $\omega(m) = \min_{2\le k\le m} |\lambda^k - \lambda|$ for any $m\ge 2$, and a quadratic polynomial $\lambda z + z^2$ is linearizable near $0$ if and only if $\lambda$ is Brjuno \cite{Brjuno1} \cite{Brjuno2} \cite{Yoccoz}. For higher degree polynomials the Brjuno condition \eqref{eq:brjuno} is sufficient, but necessity is still an open question. Similarly, we do not know whether the Brjuno condition is necessary for Theorem \ref{thm:main}, but we will show that it cannot be completely omitted.

By Sullivan's No Wandering Domains Theorem \cite{Sullivan1985} it is known that all Fatou components of a one-dimensional polynomial are periodic or pre-periodic. It is clear that the techniques used by Sullivan, which involve quasi-conformal deformations, cannot be extended to higher dimensions, not even to polynomial skew-products. There are no known proofs of Sullivan's Theorem that do not use quasi-conformal deformations. Under additional assumptions on the polynomial one can however prove the non-existence of wandering Fatou components without using quasi-conformal deformations.

If every critical point of the polynomial either (1) lies in the basin of an attracting periodic cycle, (2) lies in the basin of a parabolic periodic cycle, or (3) lies in the Julia set and after finitely many iterates is mapped to a periodic point, then one can construct a conformal metric $\mu$, defined in a backward invariant neighborhood of the Julia set minus the parabolic periodic orbits, so that $g_0$ is expansive with respect to this metric. It follows that there can be no wandering Fatou components.

The main idea in the proof of Theorem \ref{thm:main} is that for fibers $\{z = z_0\}$ sufficiently close to the invariant fiber $\{z = 0\}$ we can define conformal metrics $\mu_{z_0}$ that depend continuously on $z_0$, and so that $F$ acts expansively with respect to this family of metrics. That is, for a point $(z_0,w_0) \in \mathbb C^2$ lying in the region where the metrics are defined, and for a non-zero vertical tangent vector $\xi \in T_{w_0}(\mathbb C_{z_0})$ we have that
$$
\mu_{z_1} (w_1, dg_{z_0} \xi) > \mu_{z_0}(w_0, \xi),
$$
where $(z_1,w_1) = F(z_0, w_0)$.

We note that here we only allow critical points of the type (1) and (2). Whether the same techniques could be used to deal with pre-periodic critical points lying on the Julia set is unclear to us. The difficulty is that the property that critical points are eventually mapped onto periodic cycles is not preserved in nearby fibers. Considering similar difficulties that were overcome in \cite{PS} it is possible that expanding metrics could still be constructed on some dense set of nearby fibers, which would be sufficient for the description of Fatou components.

A local description of the dynamics near the parabolic periodic points is essential in the construction of the expanding metrics on nearby fibers. Without loss of generality we may assume that the skew-product
$$
F(z,w) = (\lambda \cdot z, f_z(w))
$$
satisfies $f_0(w) = w + w^{k+1} + O(w^{k+2})$ for some $k \ge 1$. We will prove the following.

\begin{prop}\label{localcoordinates}
Let $F$ be a holomorphic skew-product of the form
\begin{equation}
F(z,w) = (f(z),g_z(w))
\end{equation}
with $f(z) = \lambda  z + O(z^2)$, $g_0(0)=0$, and $g_0'(0)=1$. Assume $\lambda$ is a Brjuno number. If $g_0(w)\equiv w$, then $F$ is holomorphically linearizable. If $g_0(w) = w + g_{0,k+1} w^{k+1} + O(w^{k+2})$ with $g_{0,k+1}\ne0$ for some $k \ge 1$, then for any $h\ge 0$ there exists a local holomorphic change of coordinates near the origin conjugating $F$ to a map of the form $\widetilde{F}(z,w) = (\lambda z, \tilde{g}(z,w))$ satisfying
\begin{equation}
\tilde{g}(z,w) = w + g_{0,k+1} w^{k+1} + \cdots + g_{0,k+h+1} w^{k+h+1} + \sum_{j\ge h}w^{k+j+2}\alpha_{k+j+2}(z),
\end{equation}
where, for $j\ge h$, $\alpha_{k+j+2}(z)$ is a holomorphic function in $z$ such that $\alpha_{k+j+2}(0) = g_{0,k+j+2}$.
\end{prop}

In the proof of this proposition we use the assumption that $\lambda$ is Brjuno. We will also construct an Cremer-like example for $\lambda$ non-Brjuno for which the above description is false.

The organization of the paper is as follows. In section 2 we will consider the local dynamics near a parabolic point, and prove Proposition \ref{localcoordinates} above. In section 3 we show that Proposition \ref{localcoordinates} does not hold in the absence of the Brjuno condition. In fact, we prove that parabolic basins do not need to bulge in the non-Brjuno setting. In section 4 we discuss the construction of expanding metrics, and in section 5 we conclude the proof of Theorem \ref{thm:main}.

\section{Local changes of coordinates}

This section is devoted to the proof of Proposition~\ref{localcoordinates}.
We shall use a procedure recalling the Poincar\'e-Dulac normalization process \cite[Chapter 4]{arnold} aiming to conjugate the given polynomial skew-product to a skew-product in a simpler form. At each step of the usual Poincar\'e-Dulac normalization procedure we can use a polynomial change of coordinates to eliminate all non-resonant monomials of a given degree. We shall use a similar idea here, but thanks to the skew-product structure of the germ and the Brjuno assumption we shall be able to eliminate all non-resonant terms of a given degree in the powers of $w$ by means of changes of coordinates that are polynomial in $w$ with holomorphic coefficients in $z$. It turns out that there are some differences in between the first degrees.
We shall divide the proof into three steps to make such differences clearer to the reader.

Let $F\colon(\C^2,O)\to(\C^2,O)$ be a germ of skew-product of the form
\begin{equation}\label{eq_0}
F(z,w) = (\lambda z, a_0(z)+ a_1(z) w + \cdots a_d(z) w^d)
\end{equation}
where $a_0(z), \dots, a_d(z)$ are germs of holomorphic functions. Then $F$ admits a germ of local holomorphic invariant curve of the form $\{w = \phe(z)\}$ if and only if
\begin{equation}\label{eq_phi1}
a_0(z) + a_1(z)\phe(z)+\cdots + \cdots +a_d(z)\phe(z)^d = \phe(\lambda z),
\end{equation}
which is equivalent to the existence of a local holomorphic change of coordinates $\Phi$ of the form
\begin{equation}\label{eq_phi}
\Phi(z,w)
=
(z, w + \phe(z)),
\end{equation}
conjugating $F$ to a germ
\begin{equation}\label{eq_1bis}
\widetilde F(z,w) = (\lambda z, \alpha_1(z) w + \cdots+ \alpha_d(z) w^d).
\end{equation}
Thus $F$ admits a germ of an invariant holomorphic curve of the form $\{w = \phe(z)\}$ if and only if $F$ is locally holomorphically conjugated to a map of the form \eqref{eq_1bis}.

The next lemma follows as a corollary of Brjuno's proof \cite{Brjuno1, Brjuno2}, and also as a corollary of a theorem of P\"oschel \cite{Po}). We report the proof here for the sake of completeness.

\begin{lemma}\label{lemmaone}
If $\lambda$ is a Brjuno number and $a_1(0) = 1$, then $F$ admits a germ of an invariant holomorphic curve of the form $\{w = \phe(z)\}$.
\end{lemma}
\begin{proof}
First we shall search for a solution $\phe$ of \eqref{eq_phi1} in the set of formal power series, and then we shall prove such solution is indeed convergent. Write
$$
\phe(z) = \sum_{n\ge 0} \phe_n z^n,\quad a_j(z) = \sum_{n\ge 0} a_{j, n}z^n~\hbox{for}~j=0, \dots, d.
$$
Then \eqref{eq_phi1} can be written as
$$
\begin{aligned}
\sum_{n\ge 0} a_{0, n}z^n
& +
\left(\sum_{n\ge 0} a_{1, n}z^n\right)\left(\sum_{n\ge 0} \phe_n z^n\right)
+
\cdots \\
\cdots
 & +
\left(\sum_{n\ge 0} a_{d, n}z^n\right)\left(\sum_{n\ge 0} \phe_n z^n\right)^d
-
\sum_{n\ge 0} \phe_n\lambda^n z^n
=
0.
\end{aligned}
$$
Recalling that $a_{0,0} =0$ and $a_{1,0}= 1$ it follows that $\phe_0$ has to satisfy
\begin{equation}\label{polyphi}
(a_{1,0}-1)\phe_0 + a_{2,0}\phe_0^2 + \cdots + a_{d,0}\phe_0^d
=
a_{2,0}\phe_0^2 + \cdots + a_{d,0}\phe_0^d
=
0.
\end{equation}
We can therefore choose the root $\phe_0=0$ of the polynomial \eqref{polyphi}, and for any $p\ge 1$ we obtain
$$
(\lambda^p - 1) \phe_p
=
a_{0,p} + \sum_{n=1}^p a_{1,n}\phe_{p-n} + \sum_{m=2}^d \sum_{n=0}^p a_{m,n} \sum_{i_1+\cdots+ i_m = p-n} \phe_{i_1}\cdots\phe_{i_m}.
$$
Since $\lambda^p - 1 \neq 0$, these equalities give a recursive definition of the formal power series $\phe$. To show that the power series is convergent, note that since $a_0(z), \dots, a_d(z)$ are holomorphic functions on $D(0,\rho)$ for some $\rho>0$, there exists $M>0$ such that $|a_{m,n}|\le \frac{M}{\rho^n}$ for $m=1,\dots, d$ and for all $n\ge 0$. Therefore we have
\begin{equation}\label{brjuno}
|\phe_p|
\le
\frac{M}{\rho^p|\lambda^p - 1|}
\left( 1 + \sum_{m=2}^d \sum_{i_1+\cdots+ i_m = p} |\phe_{i_1}|\cdots|\phe_{i_m}|
\right),
\end{equation}
for all $p\ge 2$. These inequalities are exactly the same as the inequalities occurring in the proof of Brjuno's Theorem \cite{Brjuno1, Brjuno2}, and so, thanks to the assumption that $\lambda$ is a Brjuno number, we deduce the convergence of $\phe(z)$.
\end{proof}


\begin{lemma}\label{le_2}
Let $F$ be a germ of holomorphic skew-product of the form
\begin{equation*}
F(z,w)
=
\left(\lambda z, w + a_1(z) w + \sum_{j\ge 2} a_j(z)w^j\right)
\end{equation*}
with $a_1(0)=0$, $a_1(z)\not\equiv 0$, and $\lambda$ a Brjuno number. Then there exists $\rho_1>0$ and $\Phi_1\colon D(0, \rho_1)\times\C\to D(0, \rho_1)\times\C$ a holomorphic change of coordinates of the form
$$
\Phi_1(z,w)
=
(z, w + \psi(z)w)
$$
with $\psi(0)=0$ and such that
\begin{equation}\label{eq_3}
\Phi_1^{-1}\circ F\circ \Phi_1 (z,w)
=
\left(\lambda z, w + \sum_{j\ge 2} \alpha_j(z)w^j\right),
\end{equation}
where the $\alpha_j(z)$'s are holomorphic functions on $D(0,\rho_1)$.
\end{lemma}

\begin{proof}
Observe that $\Phi_1^{-1}(z,w)= (z, \frac{w}{1+\psi(z)})$, so it suffices to prove that we can find $\psi(z)$ holomorphic on $D(0, \rho_1)$ for $\rho_1>0$ sufficiently small such that
$\psi(0)=0$, $|\psi(z)|<1$ and
\begin{equation}\label{eq_conj_2}
\psi(z)(1+ a_1(z)) - \psi(\lambda z) +  a_1(z)= 0.
\end{equation}
As before, we shall first search for a solution of \eqref{eq_conj_2} in the set of formal power series, and then we prove that there exists $\rho_1>0$ such that the solution is indeed convergent on $D(0, \rho_1)$.
Write
$$
\psi(z) = \sum_{n\ge 1} \psi_n z^n,\quad  a_1(z) = \sum_{n\ge 1}  a_{1,n}z^n.
$$
Then \eqref{eq_conj_2} can be written as
$$
\sum_{n\ge 1} \psi_n z^n
-
\sum_{n\ge 1} \lambda^n\psi_n z^n
+
\sum_{n\ge 1}  a_{1,n}z^n
+
\left(\sum_{n\ge 1} a_{1,n} z^n\right)\left(\sum_{m\ge 1} \psi_{m}z^m\right)
=0.
$$
Hence we can set
$$
\psi_1 =\frac{ a_{1,1}}{\lambda-1},
$$
and for any $p\ge 2$ we have
$$
\psi_p
=
\frac{1}{\lambda^p - 1}
\left( a_{1,p} + \sum_{m=1}^{p-1}  a_{1,m}\psi_{p-m} \right).
$$
Since $a_1(z)$ is holomorphic on $D(0,\rho)$, there exists $M>0$ such that $| a_{1,n}|\le M/\rho^n$ for all $n\ge 1$. Therefore we have
\begin{equation}\label{brjuno_easy}
|\psi_p|
\le
\frac{M}{\rho^p|\lambda^p - 1|}
\left(1 + \sum_{m=1}^{p-1} |\psi_{m}| \right),
\end{equation}
for all $p\ge 2$, and again thanks to the assumption that $\lambda$ is a Brjuno number, we deduce the convergence of $\psi(z)$.
\end{proof}

\begin{lemma}\label{le_gen}
Let $F$ be a germ of holomorphic skew-product of the form
\begin{equation}\label{eq_gen}
F(z,w)
=
\left(\lambda z, w + a_2 w^2 + \cdots +  a_k w^k + \sum_{m\ge k+1}\alpha_m(z) w^m\right),
\end{equation}
with $\lambda$ a Brjuno number, $2\le k < +\infty$, $a_2, \dots , a_k\in \C$ and $\{\alpha_m(z)\}_{m\ge k+1}$ holomorphic functions on $D(0, r)$.
Then there exist $\rho>0$ and a holomorphic change of coordinates $\Phi$, defined on the bidisk $D(0, \rho)\times D(0, \rho)$, of the form
$$
\Psi(z,w)
=
(z, w + h(z)w^{k+1})
$$
with $h(0)=0$, such that
\begin{equation}\label{eq_gen2}
\Phi^{-1}\circ F\circ \Phi (z,w)
=
\left(\lambda z, w  + a_2 w^2 + \cdots + a_{k+1} w^{k+1} + \sum_{m\ge k+2} \beta_m(z) w^m\right),
\end{equation}
where $a_{k+1} = \alpha_{k+1}(0)$ and $\{\beta_m(z)\}_{m\ge k+2}$ are holomorphic functions on $D(0, \rho)$.
\end{lemma}

\begin{proof}
Note that if $\Psi$ exists holomorphic, we can only say that
$$
\Psi^{-1}(z,w)= (z, w - \xi(z)w^{k+1} + w^{k+2}R(z,w)),
$$
where $R(z,w)$ is holomorphic on $D(0, \rho)\times D(0, \rho)$.
It suffices to prove that we can find $\xi(z)$ holomorphic on $D(0, \rho)$ for $\rho>0$ sufficiently small such that $\xi(0)=0$ and
\begin{equation}\label{eq_conj_3}
\xi(z) - \xi(\lambda z) + \alpha_{k+1}(z)= \alpha_{k+1}(0).
\end{equation}
As before, we first search for a solution of \eqref{eq_conj_3} in the set of formal power series, and then we shall prove that there exists $\rho>0$ so that the solution is convergent on $D(0, \rho)$.
Write
$$
\xi(z) = \sum_{n\ge 1} \xi_n z^n,\quad \alpha_{k+1}(z) = a_{k+1} +\sum_{n\ge 1} a_{k+1,n} z^n.
$$
Therefore \eqref{eq_conj_3} can be written as
$$
\sum_{n\ge 1} \xi_n z^n
-
\sum_{n\ge 1} \lambda^n\xi_n z^n
+
a_{k+1}
+ \sum_{n\ge 1} a_{k+1,n}z^n
=
a_{k+1}.
$$
We can then set
$$
\xi_n =\frac{\alpha_{k+1,n}}{\lambda^n-1},
$$
for all $n\ge 1$.
Since $\alpha_{k+1}(z)$ is holomorphic on $D(0,r)$, there exists $M>0$ such that $|\alpha_{k+1,n}|\le M/r^n$ for all $n\ge 1$. Therefore we have
\begin{equation}\label{brjuno_easy_2}
|\xi_n|
\le
\frac{M}{r^n|\lambda^n - 1|},
\end{equation}
for all $n\ge 1$, and once again thanks to the assumption that $\lambda$ is a Brjuno number, we deduce the convergence of $\xi(z)$.
\end{proof}

\begin{remark}
We note that to prove the convergence of $\phe$ it suffices to have
$$
\limsup_{m\to+\infty}\frac{1}{m}\log\frac{1}{|\lambda^m-1|}<+\infty.
$$
\end{remark}

We can finally restate and prove Proposition \ref{localcoordinates}.

\medskip

\noindent {\bf Proposition \ref{localcoordinates}.} \emph{
Let $F$ be a holomorphic skew-product of the form
\begin{equation}\label{eqprop0}
F(z,w) = (f(z),g_z(w))
\end{equation}
with $f(z) = \lambda  z + O(z^2)$, $g_0(0)=0$, and $g_0'(0)=1$. Assume $\lambda$ is a Brjuno number. If $g_0(w)\equiv w$, then $F$ is holomorphically linearizable. If $g_0(w) = w + g_{0,k+1} w^{k+1} + O(w^{k+2})$ with $g_{0,k+1}\ne0$ for some $k \ge 1$, then for any $h\ge 0$ there exists a local holomorphic change of coordinates near the origin conjugating $F$ to a map of the form $\widetilde{F}(z,w) = (\lambda z, \tilde{g}(z,w))$ satisfying
\begin{equation}\label{eqprop1}
\tilde{g}(z,w) = w + g_{0,k+1} w^{k+1} + \cdots + g_{0,k+h+1} w^{k+h+1} + \sum_{j\ge h}w^{k+j+2}\alpha_{k+j+2}(z),
\end{equation}
where, for $j\ge h$, $\alpha_{k+j+2}(z)$ is a holomorphic function in $z$ such that $\alpha_{k+j+2}(0) = g_{0,k+j+2}$.
}

\medskip

\begin{proof} Thanks to the hypothesis that $\lambda$ is a Brjuno number, there exists a local change of variables $\phe_f\colon D(0, \rho)\to D(0, \rho)$, where $\rho>0$, such that $\phe_f(0)=0$, $\phe_f'(0)=1$ and $\phe_f^{-1}\circ f\circ \phe_f (z) = \lambda z$.
Therefore, up to conjugating via the holomorphic change of coordinates $\Phi_0\colon D(0, \rho)\times\C\to D(0, \rho)\times\C$ defined by $\Phi_0(z,w) = (\phe_f(z), w)$, for all $(z,w)\in D(0, \rho)\times\C$ we may assume that $F$ is of the form
\begin{equation}\label{eq_1}
F(z,w)
=
\left(\lambda z, \sum_{j\ge 0} a_j(z)w^j\right)
\end{equation}
where the $a_j(z)$'s are holomorphic functions on $D(0,\rho)$, $a_0(0)=0$ and $a_1(0) = 1$.

\medskip

If $g_0(w) \equiv w$, then $a_j(0) =0$ for all $j\ge 2$, $F$ satisfies the hypotheses of \cite[Theorem 1.11]{jasmin}, and hence it is holomorphically linearizable.

\medskip

If $g_0(w) = w + g_{0,k+1} w^{k+1} + O(w^{k+2})$ with $g_{0,k+1}\ne0$ for some $k \ge 1$, then
$a_1(0)=1$, $a_0(0)=a_2(0) =\cdots= a_k(0)=0$, and $a_{k+1}(0)=  g_{0,k+1} \ne 0$. The rest of the proof follows from Lemma \ref{lemmaone}, Lemma \ref{le_2} and Lemma \ref{le_gen}.
\end{proof}

%
%
%
%
%
%
%
%
%
%
%
%
%

\section{Non-Brjuno rotations}\label{section:nonBrjuno}

If $\lambda$ is elliptic and not Brjuno, Proposition \ref{localcoordinates} does not hold in general. In fact it is possible to construct examples arguing similarly to Cremer \cite{Cremer}.

\begin{defn}
A complex number $\lambda\in\C$ with $|\lambda|=1$ is called {\it Cremer} if
$$
\limsup_{m\to\infty} \frac{1}{m} \log\frac{1}{\omega(m)} = +\infty,
$$
where $\omega(m) = \min_{2\le k\le m} |\lambda^k - \lambda|$.
\end{defn}

We have the following elementary example.

\begin{example}
Let $\lambda$ be a Cremer number. Then the polynomial skew-product
$$
F(z,w) = (\lambda z, w + z + zw)
$$
does not admit a holomorphic invariant curve of the form $\{w =\phe(z)\}$.
In fact in this case the computation for a formal solution $\phe(z) = \sum_{n\ge0}\phe_n z^n$ as in Lemma \ref{lemmaone} gives the explicit formulas
\begin{equation*}
\begin{aligned}
\phe_0&\in\C,\\
\phe_1&= \frac{1+ \phe_0}{\lambda -1},\\
\phe_n &= \frac{\phe_{n-1}}{\lambda^n - 1} = \prod_{j=1}^n\frac{1+ \phe_0}{\lambda^j-1},~\hbox{for}~n\ge 2.
\end{aligned}
\end{equation*}
Therefore
\begin{equation*}
\begin{aligned}
\limsup_{m\to\infty} \frac{1}{m} \log|\phe_m|
&= \limsup_{m\to\infty} \frac{1}{m}\log|1+\phe_0|
+ \limsup_{m\to\infty} \frac{1}{m}\sum_{j=1}^m \log \frac{1}{|\lambda^j-1|}\\
&=\limsup_{m\to\infty} \frac{1}{m}\log \frac{1}{\displaystyle\min_{1\le k\le m} |\lambda^k - 1| }+ \limsup_{m\to\infty} \frac{1}{m}\sum_{j=1\atop j\ne j_m}^m \log \frac{1}{|\lambda^j-1|},
\end{aligned}
\end{equation*}
where $j_m\in\{1,\dots m\}$ is such that $\min_{1\le k\le m} |\lambda^k - 1| = |\lambda^{j_m}-1|$. And hence
$$
\limsup_{m\to\infty} \frac{1}{m} \log|\phe_m|
\ge\limsup_{m\to\infty} \frac{1}{m} \log\frac{1}{\omega(m)} + \limsup_{m\to\infty} \frac{m-1}{m}\log\frac{1}{2} = +\infty.
$$
Note that if $\displaystyle\limsup_{m\to\infty}\log\frac{1}{\omega(m)} < +\infty$, then $\phe$ is convergent.
\end{example}

Arguing as in Cremer's example \cite{Cremer} we can also show the existence of holomorphic skew-products without invariant holomorphic curves of the form $\{w=\varphi(z)\}$, having an elliptic invariant fiber that is not point-wise fixed.

\begin{prop}
Let $\lambda$ be a Cremer number. Then there exists a skew-product of the form
$$
F(z,w) = (\lambda z, w + a(z) + w^2),
$$
with $a(z)$ holomorphic and satisfying $a(0)=0$, such that $F$ does not admit a holomorphic invariant curve of the form $\{w =\phe(z)\}$.
\end{prop}

\begin{proof}
Consider $a(z) = \sum_{n\ge 1} a_n z^n$ the power series expansion of $a(z)$ near the origin.
In this case the computation for a formal solution $\phe(z) = \sum_{n\ge 0}\phe_n z^n$ as in Lemma \ref{lemmaone} gives
\begin{equation*}
\begin{aligned}
\phe_0&=0,\\
\phe_1&= \frac{a_1}{\lambda -1},\\
\phe_n &= \frac{1}{\lambda^n - 1}\left(a_n + \sum_{j=1}^{n-1} \phe_j\phe_{n-j}\right)~\hbox{for}~n\ge 2.
\end{aligned}
\end{equation*}
We can then set $a_1=1$ and recursively choose $a_n\in\{0,1\}$ such that
$$
\left|a_n + \sum_{j=1}^{n-1} \phe_j\phe_{n-j}\right|
\ge
\frac{1}{2}
$$
for $n\ge 2$.
Therefore
\begin{equation*}
\begin{aligned}
\limsup_{m\to\infty} \frac{1}{m} \log|\phe_m|
&= \limsup_{m\to\infty} \frac{1}{m}\log\left|\frac{1}{\lambda^n - 1}\left(a_n + \sum_{j=1}^{n-1} \phe_j\phe_{n-j}\right)\right| \\
&\ge \limsup_{m\to\infty} \frac{1}{m}\log\frac{1}{2} + \limsup_{m\to\infty} \frac{1}{m}\log\frac{1}{|\lambda^n - 1|}\\
&\ge\limsup_{m\to\infty} \frac{1}{m} \log\frac{1}{\omega(m)} = +\infty,
\end{aligned}
\end{equation*}
and this concludes the proof.
\end{proof}

The same approach can be used to construct polynomial skew-products of higher degrees giving counter-examples to Proposition \ref{localcoordinates}. Moreover we have the following result.

\begin{prop}
Let $F(z,w) = (f(z), g_z(w))$ be a holomorphic skew-product with an elliptic linearizable invariant fiber. If $F$ does not admit a holomorphic invariant curve on the invariant fiber, then the parabolic Fatou components of $g_0$ do not bulge.
\end{prop}

\begin{proof}
Up to a change of coordinates in a neighborhood of the invariant fiber, we can assume that $F$ is of the form
$$
F(z,w) = (\lambda z, g_z(w)) = (\lambda z, a_0(z) + a_1(z) w + \cdots + a_d(z)w^d).
$$
Assume by contradiction that there is a parabolic Fatou component of $g_0(w)$ bulging to a $2$-dimensional Fatou component $U$. Then there exists a subsequence $(n_j)\subset\N$ such that the maps $F^{n_j}$ must converge uniformly on compact subsets of $U$, say to a map $h: U \rightarrow \mathbb C$. The map $h$ has generic rank $0$, $1$ or $2$. Since the invariant fiber is linearizable, the bulging component cannot be contracted to a point, and thus the rank of $h$ cannot be $0$. Since the component of $g_0$ is parabolic, all orbits starting in $U$ must converge to the Julia set of $F$, and the generic rank of $h$ cannot be $2$. Thus the rank of $h$ must be $1$. It was shown in \cite{LP2014} that the analytic set $h(U)$ must in fact be an injectively immersed Riemann surface.

By passing to an iterate of $F$ we can assume that the one-dimensional component on the invariant fiber is invariant or pre-invariant. It therefore follows that the Fatou component $U$ must be invariant or pre-invariant as well, hence the limit set $h(U)$ must be invariant. Note that $h(U)$ must contain the corresponding parabolic fixed point of $g_0$, and is locally given as a graph $\{w = \phi(z)\}$. Since the invariant fiber is linearizable, this graph must contain an invariant holomorphic disk, which contradicts the hypothesis on $F$.
\end{proof}

\section{Expanding horizontal metrics}

In this section we will construct a continuously varying family of conformal metrics on the planes $\{z = z_0\}$, for $z_0$ sufficiently small. Before doing so we discuss the dynamics near the invariant curves through the parabolic points of $g_0$.  Let us recall that the degree of $g_w$ is assumed to be constant near $\{z = 0\}$, and at least $2$. Thus $g_0(w) \not\equiv w$. We consider the dynamics near a parabolic periodic point of $g_0$, which without loss of generality may be assumed to be the fixed point $w=0$.

It follows from Proposition \ref{localcoordinates} that, after a local holomorphic changes of coordinates, we can assume that, for $(z,w)\in D(0, \rho)^2$, $F$ is of the form
\begin{equation}
F(z,w)
=
\left(\lambda z, w  - w^{k+1} + a_{k+2}w^{k+2} + \cdots +a_{2k+1}w^{2k+1} +  \sum_{m\ge 2k+2} \alpha_m(z) w^m\right),
\end{equation}
where $k\ge 1$ and the functions $\{\alpha_m(z)\}_{m\ge 2k+2}$ are holomorphic on $D(0, \rho)$.
Hence, up to a polynomial change of coordinate of the form $(z,w)\mapsto (z, w + q(w))$ with $q(0)= q'(0)= 0$, we can assume
\begin{equation}\label{eq_gen_start_red2}
F(z,w)
=
\left(\lambda z, w  - w^{k+1} + b w^{2k+1} +  \sum_{m\ge 2k+2} \beta_m (z) w^m\right),
\end{equation}
where $k\ge 1$, $b$ is the index of $g_0(w)$ at the origin, and $\{\beta_m(z)\}_{m\ge 2k+2}$ are holomorphic functions on $D(0, \rho)$.

Denote by $P^+(\rho, \eta)$ the union of the attracting petals of $w\mapsto w-w^{k+1}$ of radius $\rho>0$ and wideness $\eta>0$. It follows that for $F$ of the form \eqref{eq_gen_start_red2} the open set
$$
B_{\rho,\eta} = \{(z,w)\in D(0, \rho)^2\mid w\in P^+(\rho, \eta)\}
$$
is forwards invariant under $F$ for $\rho$ and $\eta$ sufficiently small. Indeed, we can use the fact that we have a uniform bound for $|\beta_m(z)|$ if $|z|<\rho$ is small, in order to argue as in the proof of the Leau-Fatou Flower Theorem, see for example \cite{Milnor}. Moreover, if $(z_0, w_0)\in B_{\rho,\eta}$, then $w_n = \pi_2(F^n(z_0,w_0))$ converges to 0 tangent to an attracting direction of the function $w\mapsto w-w^{k+1}$. The set $B_{\rho, \eta}$ is an absorbing domain, in the sense that the orbit of any point $(z,w)$, with $|z|< \rho$, that lies in one of the bulging parabolic Fatou components corresponding to the parabolic fixed point $w =0$ of $g_0$ must eventually enter the set $B_{\rho, \eta}$.

\medskip

Let us now recall the one-dimensional construction of the expanding metric for the polynomial $p = g_0$. A more thorough discussion of such metrics can be found in \cite{CG1993}. By assumption all critical points of $p$ are contained in the basins of attracting and parabolic periodic cycles. As there are only finitely many of such cycles, by passing to a large iterate of $F$ we may assume that all periodic attracting and parabolic cycles are in fact fixed points.

Define $D_0 \subset \mathbb C_0$, a closed set containing a forward invariant neighborhood of each attracting fixed point, the parabolic fixed points, and the attracting petals of each parabolic fixed point. Then there exists a positive integer $N$ so that $D = p^{-N}(D_0)$ contains all critical points. Note that $D$ is closed and forward invariant, hence its complement $U$ is open and backwards invariant. Let $\mu_U$ be the hyperbolic metric on the complement of $D$. Given any neighborhood of the parabolic fixed points we can choose $N$ sufficiently small so that $p$ is expanding with respect to $\mu_U$ except for points lying in these parabolic neighborhoods.

Let $B$ be the union of the repelling parabolic petals, chosen sufficiently small, of all the parabolic fixed points. Define the metric $\mu_B$ near each parabolic fixed point by $\mu_B = |d\zeta|$, using the local coordinates
$$
p: \zeta \mapsto \zeta + \zeta^{k+1} + O(\zeta^{k+2}).
$$
Now set $\mu = \inf\{\mu_U, M \cdot \mu_B\}$ on the union $A = U \cup B$, where $M>0$ is a constant. By first choosing $N \in \mathbb N$ sufficiently large so that $p$ is expanding on $D \setminus B$, and then choosing $M>0$ sufficiently large, the map $p$ will be strictly expanding on $D$ with respect to the metric $\mu$.

\medskip

The above construction can be carried out in each fiber $\{z = z_0\}$, for $z_0$ sufficiently close to $0$. Each neighborhood of an attracting cycle will still be mapped relatively compactly into itself by $g_{z}$ for $z$ sufficiently small. Hence we can take the union of these neighborhoods, the invariant curves through the parabolic fixed points, and the parabolic ``petals'' $B_{\rho,\eta}$ discussed earlier in this section, and map this closed forwards invariant set backwards by $F^N$ for $N\in \mathbb N$, and denote the complement in $\{|z|< \rho\}$ by $U$. By construction it follows that for $N$ sufficiently large, and by decreasing $\rho>0$ if necessary, the open and backwards invariant set $U$ does not intersect the critical set $\{\frac{\partial F}{\partial w} = 0\}$.

We denote the intersection of $U$ with $\{z = z_0\}$ by $U(z_0)$ and write $\mu_{U(z_0)}$ for the hyperbolic metric on $U(z_0)$. By the backwards invariance of $U$ we have that $U(z_0) \subset F(U(\lambda^{-1}z_0))$. Similarly to the attracting parabolic petal $B_{\rho, \eta}$ we can define the repelling parabolic petal, which is backwards invariant. Recalling the local coordinates
\begin{equation}\label{eq_gen_start_red}
F(z,\zeta)
=
\left(\lambda z, \zeta  - \zeta^{k+1} + b \zeta^{2k+1} +  \sum_{m\ge 2k+2} \beta_m (z) \zeta^m\right),
\end{equation}
we immediately see that on the repelling parabolic petals, which in these local coordinates can be defined independently of $z$, $F$ acts expandingly in the vertical direction with respect to the metrics $\mu_{B(z_0)} = |d\zeta|$. Thus, just as in the one-dimensional setting, we can first choose $N$ sufficiently large and then $M$ sufficiently large so that $F$ acts expandingly with respect to the metrics
$$
\mu(z_0) = \inf\{\mu_{U(z_0)}, M \cdot \mu_{B(z_0)}\}.
$$
That is,
$$
\mu_{z_1} (w_1; dg_{z0}\xi) > \mu_{z_0} (w_0; \xi)
$$
whenever $(z_1, w_1) = F(z_0, w_0) \in U$.

\section{Fatou components near the invariant fiber}

Recall that we are dealing with a skew-product of the form
$$
F(z,w) = (\lambda \cdot z, g_z(w)),
$$
where $\lambda$ is assumed to satisfy the Brjuno condition, and $g_z$ is a polynomial in $w$ with coefficients depending holomorphically on $z$. Let us restate and prove our main result.

\medskip

\noindent{ \bf Theorem 1.}
\emph{If $\lambda$ is Brjuno and all critical points of the polynomial $g_0$ lie in basins of attracting or parabolic cycles, then all Fatou components of $g_0$ bulge, and there is a neighborhood of the invariant fiber $\{z=0\}$ in which the only Fatou components of $F$ are the bulging Fatou components of $g_0$. In particular there are no wandering Fatou components in this neighborhood.}

\medskip

\begin{proof} In order to show that all Fatou components of $g_0$ bulge, it is sufficient to show that the periodic Fatou components of $g_0$ bulge. Every periodic Fatou of $g_0$ is either an immediate basin of an attracting periodic point, or an immediate basin of a parabolic periodic cycle. The attracting basins always bulge. If $\lambda$ is Brjuno, then the parabolic basins bulge, as was shown in Proposition \ref{localcoordinates}. In order to complete the proof of Theorem \ref{thm:main} it therefore remains to be shown that $F$ has no other Fatou components in a neighborhood of the invariant fiber.

As before we may assume that all parabolic cycles are fixed points by considering a high iterate of $F$. Recall that the local change of coordinates found in Proposition \ref{localcoordinates} near each of the parabolic fixed points gives an invariant holomorphic disk $\{w = \phi_i(z)\}$, which we refer to as a parabolic disk. Since there are only finitely many parabolic fixed points, we can find a $\delta>0$ so that the parabolic disks are all defined for $|z|< \delta$. We have obtained metrics $\mu_z$ on each of those fibers, so that the action of $F$ with respect to these metrics is strictly expansive. Moreover, given any neighborhood of the parabolic disks the metrics are \emph{uniformly} expanding outside of this neighborhood. Now note that any point whose orbit converges to a parabolic disk must lie in the bulging Fatou component corresponding to a parabolic basin of $g_0$.

Now suppose that $0<|z_0|<\delta$ and $(z_0,w_0)$ does not lie in one of the bulging Fatou components. Then the orbit $(z_n,w_n)$ remains in the neighborhood where the expanding metrics are defined, and must infinitely often visit the complement of some neighborhood of the parabolic disks. Hence for any vertical tangent vector $\xi \in T_{w_0}(\mathbb C_{z_0})$ we have that
$$
\mu_{z_n} dg_{w_{n-1}} \cdots dg_{w_0} \xi \rightarrow \infty,
$$
from which it follows that the family ${F^n}$ cannot be normal on any neighborhood of $(z_n,w_n)$. This proves that in the region $\{|z|<\delta\}$ there are no Fatou components but the bulging components.
\end{proof}

\end{document}